\input ./style/arxiv-vmsta.cfg
\documentclass[numbers,compress,v1.0.1]{vmsta}
\usepackage{vtexbibtags}
\usepackage{authorquery}

\volume{5}% Updated by VTEXPTS2LaTeX.exe, 18.12.2018 09:32
\issue{4}% Updated by VTEXPTS2LaTeX.exe, 18.12.2018 09:32
\pubyear{2018}
\firstpage{509}% Updated by VTEXPTS2LaTeX.exe, 18.12.2018 09:32
\lastpage{519}% Updated by VTEXPTS2LaTeX.exe, 18.12.2018 09:32
\aid{VMSTA108}
\doi{10.15559/18-VMSTA108}% Updated by VTEXPTS2LaTeX.exe, 04.07.2018
\articletype{research-article}

\startlocaldefs

\allowdisplaybreaks

\newtheorem{proposition}{Proposition}[section]

\newtheorem{corollary}{Corollary}[section]
\theoremstyle{definition}

\newtheorem{remark}{Remark}[section]

\endlocaldefs

\begin{aqf}
\querytext{Q1}{I see that the title of the paper should be slightly modified to be in line with the main text.}
\querytext{Q2}{This sentence seems to be incomplete. Please extend and complete it.}
\querytext{Q3}{Note: The term 'non-finite divisibility' is somehow awkward. By the usual meaning of words it would mean 'finite divisibility' what, perhaps, is not what it means in fact. 'Non-finite divisibility' is just a shorthand term for the property of 'not being infinite divisible'. If so, this could be said in the text for clarity. 'Non-divisibility' could be another, convenient for reading, shorthand, but again it should be explained that 'non-divisibility' means the property of 'not being infinite divisible'. This note applies to all the text. }
\querytext{Q4}{Note: The abbreviation 'ID' is not used consistently.}
\querytext{Q5}{Why these properties are additional? Aren't they just properties?}
\querytext{Q6}{Note: "derived from their self-similarity" seems to be better than "obtained on the cost of self-similarity".}
\querytext{Q7}{Again, here actually the property of "being not infinitely divisible" is discussed.}
\end{aqf}
\begin{document}

\begin{frontmatter}
\pretitle{Research Article}

\title{On the infinite divisibility of distributions of some inverse
subordinators}

\author[a]{\inits{A.}\fnms{Arun}~\snm{Kumar}}%\ead[label=e1]{}} %
\author[b]{\inits{E.}\fnms{Erkan}~\snm{Nane}\ead[label=e2]{ezn0001@auburn.edu}\thanksref{cor1}}
%} %
\thankstext[type=corresp,id=cor1]{Corresponding author.}
\address[a]{Department of Mathematics, \institution{Indian Institute
of Technology Ropar},
Punjab 140001, \cny{India}}
\address[b]{Department of Mathematics and Statistics,
\institution{Auburn University}, Auburn, AL 36849, \cny{USA}}

%\thankstext[id=f1]{}

%\dedicated{}

%\markboth{Authors}{Title}
\markboth{A. Kumar, E. Nane}{On infinite divisibility of the
distribution of some inverse subordinators}

\begin{abstract}
\querymark{Q1}We consider the infinite divisibility of distributions of some well-known inverse subordinators. Using a tail probability bound, we
establish that distributions of many of the inverse subordinators
used in the literature are not infinitely divisible. We further show
that the distribution of a renewal process time-changed by an inverse
stable subordinator is not infinitely divisible, which in particular
implies that the distribution of the fractional Poisson process is not
infinitely divisible.
\end{abstract}
%
%\begin{keywords}[MSC2010]%
%\kwd{}
%\kwd{}
%\kwd{}
%\end{keywords}

\received{\sday{18} \smonth{3} \syear{2018}}% Updated by
%VTEXPTS2LaTeX.exe, 04.07.2018 16:44
\revised{\sday{21} \smonth{6} \syear{2018}}% Updated by
%VTEXPTS2LaTeX.exe, 04.07.2018 16:44
\accepted{\sday{29} \smonth{6} \syear{2018}}% Updated by
%VTEXPTS2LaTeX.exe, 04.07.2018 16:44
\publishedonline{\sday{20} \smonth{7} \syear{2018}}
\end{frontmatter}

%s1 #&#
\section{Introduction}
Infinitely divisible (ID) distributions were introduced by de Finetti in
1929. Ever since the research literature on these distributions is
growing rapidly. {A real-valued random variable $X$ with a cumulative
distribution function $F$ is said to be ID if for
each $n>1$, there exist independent identically distributed random
variables $X_1, X_2, \ldots, X_n$ with a distribution function $F_n$
such that
\[
X \stackrel{d}= X_1 +X_2+\cdots X_n.
\]
Well-known examples of ID distributions are
normal, Poisson, exponential, $t$, $\chi^2$ and gamma distributions.
Those that are not ID include half normal, discrete
normal, inverse normal and inverse $t$ distributions.}
ID distributions play a central role in the theory of
L\'evy processes. Note that every continuous-time L\'evy process has
distributions that are necessarily ID, and conversely every ID
distribution generates uniquely a L\'evy process (see Steutel and Van
Harn, \cite{SteHar2004}). Further, in several real life situations  some models
require a random effect to be the sum of several independent random
components\querymark{Q2} with the same distribution. In such situations a convenient
way is to assume infinite divisibility of the distribution of these
random effects. Such situations occur in biology, economics and
insurance. It is worth to mention here that to prove or disprove
infinite divisibility of a certain distribution is sometimes a very
tedious task and it may need an utterly specialized approach. In this article,
we only talk about the infinite divisibility of distributions of
some selected processes that are studied recently in the literature.\looseness=-1

 In recent years time-changed stochastic processes are getting
increased attention due to their applications in finance, geophysics,
fractional partial differential equations and in modeling the anomalous
diffusion in statistical physics (see Janczura et al., \cite{Janetal2011}; Meerschaert et
al., \cite{Meeetal2009,Meeetal2013}; Orsingher and Beghin, \cite{OrsBeg2009}). A time-changed stochastic
process is obtained by changing the time of the process by another
stochastic process. The processes that are used as time-change are
generally subordinators, or inverse subordinators. Subordinators are
non-decreasing L\'evy processes i.e. processes with independent and
stationary increments having non-decreasing sample paths. Well-known subordinators are the Poisson process, the compound Poisson processes,
the gamma process, the inverse Gaussian process, an $\alpha$-stable subordinator
and a tempered $\alpha$-stable subordinator. The first-passage time
process of a subordinator is called an inverse subordinator. For
example, the first-passage times of stable and tempered stable subordinators
are called inverse stable and inverse tempered stable subordinators,
respectively (see, e.g., Meerschaert and Straka, \cite{MeeStr2013}; Kumar and
Vellaisamy, \cite{KumVel2015}). The most popular inverse subordinator is the inverse
$\alpha$-stable subordinator (ISS). Note that ISS is used as a
time-change in the standard Poisson process to define the fractional Poisson
process (see, e.g., Meerschaert et al., \cite{Meeetal2011}; Repin and Saichev, \cite{RepSai2000};
Laskin, \cite{Las2003}; Beghin and Orsingher, \cite{BegOrs2009}). Further, ISS is used as a
time-change with the Brownian motion and a stable process to solve fractional
diffusion equations with a fractional derivative in time and fractional
derivatives both in time and space, respectively (see, e.g., Meerschaert
et al., \cite{Meeetal2009}). The time-change with a subordinator $Y(t)$--$X(Y(t))$-- is
done in the Bochner sense and results in a L\'evy process if the
process $X(t)$ is a L\'evy process.

In this article we study the infinite divisibility of the distribution
of some inverse subordinators corresponding to drift-less
subordinators. We first obtain a bound on the tail probability of these
inverse subordinators. We establish that the distributions of inverse
stable, inverse tempered stable and first-exit times of inverse
Gaussian subordinators are not ID. Further, we also show that the
distribution of a renewal process time-changed by ISS is not ID. In
particular we establish that the distribution of the fractional Poisson
process is not ID.

One should not conclude from these results that the distributions of
inverse subordinators are not ID in general. One counter-example is
the Poisson process. Let $N(t)$ be the Poisson process with rate $\lambda$.
Then the process defined by $T_n =\inf\{t\geq0: N(t)\geq n\}$,
$n=1,2\cdots$ is called the inverse of the Poisson process. For a fixed
$n$, the random variable $T_n$ is an Erlang random variable of order
$n$, with the probability density function
\[
f_{T_n}(t) = \frac{}{} \frac{\lambda e^{-\lambda t}(\lambda t)^{n-1}}{(n-1)!},\quad n=1,2,\ldots
\]
Note that the Erlang distribution is a special case of the gamma distribution
and hence the inverse of the Poisson process $N(t)$ is ID
(see, e.g., Steutel and Van Harn, \cite{SteHar2004}).

Further, the fractional Poisson process, for which applications are
suggested in insurance (Biard and Saussereau, \cite{BiaSau2014}), may not be
appropriate in situations where one needs to divide the total number of
claims in a year (say) in small intervals like months and days with independent identically distributed (i.i.d.)
components due to its \querymark{Q3}non-infinite divisibility.

ID\querymark{Q4} distributions are at the heart of the theory of L\'
evy processes. Every continuous-time L\'evy process has distributions
that are necessarily ID (see, e.g., Steutel and Van Harn, \cite{SteHar2004}; Sato,
\cite{Sat1999}). It is well known in the literature that the inverse stable
subordinator $E(t),\; t\geq0$, doesn't possess independent and
stationary increments and hence is not a L\'evy process (see
Meerschaert and Scheffler, \cite{MeeSch2004}). Our results conclude that it is not possible
even to define a continuous time L\'evy process corresponding to the
distributions of $E(1)$.

\setcounter{equation}{0}
%s2 #&#
\section{Tail probability estimates of inverse subordinators}
A subordinator is a one-dimensional L\'evy process that is
non-decreasing almost sure (a.s.). Such processes can be thought of as a random
model of time evolution. If $T(t)$ is a subordinator, then we have
%
%e1 #&#
\begin{equation}
\label{laplace-symb} \mathbb{E} \bigl(e^{-uT(t)} \bigr) = e^{-t\psi(u)},
\end{equation}
where $\psi(u)$ is called the Laplace exponent and have the following form
(see, e.g., Applebaum, \cite{App2009}, p. 53)
%
%e2 #&#
\begin{equation}
\psi(u) = bu + \int_{0}^{\infty}\bigl(1-e^{-uy}
\bigr)\nu(dy).
\end{equation}
The pair $(b, \nu)$ is called characteristics of the subordinator $T$
and represents the drift and the L\'evy measure respectively. Here we require
$\int_{0}^{\infty}(1\wedge|y|)\nu(dy)<\infty$. In this article
henceforth we only discuss subordinators with $b=0$, also called
driftless subordinators. For a subordinator $T(t)$, the first-exit time
process is defined by
%
%e3 #&#
\begin{equation}
E(t) = \inf\bigl\{s\geq0: T(s)>t\bigr\},
\end{equation}
and we call this process the inverse subordinator. Note that
%
%e4 #&#
\begin{align}
\label{tail-bound} \mathbb{P}\bigl(E(t)>x\bigr) &= \mathbb{P}\bigl(T(x) \leq t\bigr) =
\mathbb{P}\bigl( -uT(x) \geq-ut\bigr),\quad  u>0
\nonumber
\\
& = \mathbb{P}\bigl(e^{-uT(x)} \geq e^{-ut}\bigr) \leq
\frac{\mathbb
{E}e^{-uT(x)}}{e^{-ut}}\quad \mbox{(by the Markov inequality)}
\nonumber
\\
& = e^{ut-x\psi(u)}\quad \mbox{(using (\ref{laplace-symb}))}.
\end{align}
Also note that for $b=0$, $\psi'(u) = \int_{0}^{\infty}x e^{-ux}\nu
(dx)$. Further, by the dominated convergence theorem $\psi'(u)\downarrow
0$ as $u\uparrow\infty$ and hence $\psi'$ is invertible. Inequality
\eqref{tail-bound} is true for all $u>0$, and hence we can obtain a
unique upper bound. It is reached at $u$ such that
\begin{equation*}
\frac{d}{du} \bigl[e^{ut - x\psi(u)}\bigr] = 0\; \implies\; u =
\psi'^{-1}(t/x).
\end{equation*}
Thus, we have the following proposition.

%p2.1 #&#
\begin{proposition}
The tail probabilities for inverse subordinators satisfy
%
%e5 #&#
\begin{equation}
\label{tail-prob} \mathbb{P}\bigl(E(t)>x\bigr) \leq e^{t \psi'^{-1}(t/x) - x \psi(\psi
'^{-1}(t/x))},\quad  {\text{for large}}\ x.
\end{equation}
\end{proposition}

\setcounter{equation}{0}
%s3 #&#
\section{Infinite divisibility of distributions of some inverse
subordinators}

To prove the non-infinite divisibility of inverse subordinators in this
article, we use the tail bound \eqref{tail-prob} and a necessary condition
for infinite divisibility which is mentioned here (see, e.g., Steutel, \cite{Ste1979}):
A necessary condition for a cumulative distribution function
$F(x)$ to be ID is
%
%e1 #&#
\begin{equation}
\label{Steutel-VanHarn} {-\log\bigl(1-F(x)\bigr)} \leq{a x\log x},
\end{equation}
for some $a>0$ and $x$ sufficiently large.

%p3.1 #&#
\begin{proposition}[Inverse stable subordinator]\label{non-id-iss}
Let $S_{\alpha}(t)$ be an $\alpha$-stable subordinator with $\alpha
\in(0,1)$. Then the distribution of ISS defined by $E_{\alpha}(t) =
\inf\{s\geq0: S_{\alpha}(s)>t\}$ is not ID.
\end{proposition}
\begin{proof}
For an $\alpha$-stable subordinator the Laplace exponent is given by
$\psi(u) = u^{\alpha}$. Hence, we have $\psi'(u) = \alpha u^{\alpha
-1}$, which implies $\psi'^{-1}(u) =  (\frac{u}{\alpha}
)^{\frac{1}{\alpha-1}}$. Further,\break $\psi(\psi'^{-1}(u)) =
(\frac{u}{\alpha})^{\frac{\alpha}{\alpha-1}}$. Thus for
large $x$
%
%e2 #&#
\begin{equation}
\mathbb{P}\bigl(E_{\alpha}(t)>x\bigr) \leq e^{t (\frac{t}{\alpha x}
)^{\frac{1}{\alpha-1}} - x  (\frac{t}{\alpha x} )^{\frac
{\alpha}{\alpha-1}}}.
\end{equation}
Further,
%
%e3 #&#
\begin{align}
-\log\mathbb{P}\bigl(E_{\alpha}(t)>x\bigr) &\geq x \biggl(
\frac{t}{\alpha
x} \biggr)^{\frac{\alpha}{\alpha-1}} - t \biggl(\frac{t}{\alpha
x}
\biggr)^{\frac{1}{\alpha-1}}
\nonumber
\\
& = (1-\alpha) \biggl(\frac{\alpha}{t} \biggr)^{\frac{\alpha
}{1-\alpha}}
x^{\frac{1} {1-\alpha}} = d(\alpha,t)x^{\frac{1} {1-\alpha }}\quad
\mbox(\textrm{say}),
\end{align}
where $d(\alpha,t)= (1-\alpha) (\frac{\alpha}{t}
)^{\alpha/(1-\alpha)}>0$. We have,
\begin{align*}
\lim_{x\rightarrow\infty}\frac{-\log\mathbb{P}(E_{\alpha
}(t)>x)}{x\log x} & \geq\lim_{x\rightarrow\infty}
\frac{d(\alpha
,t) x^{\frac{1}{1-\alpha}}}{x\log x}
\\
& = \lim_{x\rightarrow\infty}\frac{d(\alpha,t)
x^{\frac{\alpha }{1-\alpha}}}{\log x}\quad  \bigl(\mbox{indeterminate
$\frac{\infty}{\infty}$ form}\bigr)
\\*
& = \lim_{x\rightarrow\infty} d(\alpha,t)\frac{\alpha}{(1-\alpha
)} x^{\frac{\alpha} {1-\alpha}} = \infty.
\end{align*}
Hence, a finite $a>0$ that satisfies equation \eqref
{Steutel-VanHarn} does not exist. Therefore the distribution of $E_{\alpha}(t)$ is
not ID.
\end{proof}

%r3.1 #&#
\begin{remark}
It is worthwhile to mention the results about $E_\alpha(t)$ from
Meerschaert and Scheffler \cite{MeeSch2004}. They showed that the increments of
$E_\alpha(t)$ are neither stationary nor independent.
\end{remark}
Next we prove the non-infinite divisibility of  distributions of
inverse tempered stable subordinators (ITSS). Tempered stable
subordinators (TSS) are obtained by exponential tempering in
distributions of stable subordinators (see, e.g., Rosi\'{n}ski, \cite{Ros2007}).
%Tempered stable subordinators (TSS) are obtained by exponential
%tempering in the distribution of stable subordinators.
TSS have ID distributions, have exponentially
decaying tail probabilities and have all moments finite, unlike stable
subordinators for which tail probabilities decay polynomially and first
moments are infinite. These properties\querymark{Q5} of TSS are derived
from their self-similarity.
Let $S_{\alpha, \lambda}(t)$ be the TSS with index $\alpha\in(0,1)$
and tempering parameter $\lambda>0$. The Laplace transform (LT) of the
density of TSS (see Meerschaert et al., \cite{Meeetal2013}) is
%
%e4 #&#
\begin{equation}
\label{tem-LT} \mathbb{E}\bigl(e^{-uS_{\alpha, \lambda}(t)}\bigr) = e^{-t ((u+\lambda
)^{\alpha}-\lambda^{\alpha} )}.
\end{equation}
TSS are also known as relativistic stable subordinators.
%
%p3.2 #&#
\begin{proposition} [ITSS]\label{non-id-itss}
The distributions of ITSS defined by $E_{\alpha, \lambda}(t) = \inf\{
s\geq0: S_{\alpha,\lambda}(s)>t\}$ are not ID.
\end{proposition}
\begin{proof}
The Laplace exponent for ITSS is given by $\psi(u) = (u+\lambda
)^{\alpha} - \lambda^{\alpha}$. This implies $\psi'^{-1}(u) =
(\frac{u}{\alpha} )^{\frac{1}{\alpha-1}}-\lambda$, $\psi
(\psi'^{-1}(u)) =  (\frac{u}{\alpha} )^{\frac{\alpha
}{\alpha-1}} - \lambda^{\alpha}$. Thus
%
%e5 #&#
\begin{equation}
\mathbb{P}\bigl(E_{\alpha,\lambda}(t)>x\bigr) \leq e^{-\lambda t+ t
(\frac{t}{\alpha x} )^{\frac{1}{\alpha-1}} - x  (\frac
{t}{\alpha x} )^{\frac{\alpha}{\alpha-1}} + \lambda^{\alpha}x}.
\end{equation}
Hence,
%
%e6 #&#
%e7 #&#
\begin{align}
-\log\mathbb{P}\bigl(E_{\alpha, \lambda}(t)>x\bigr) &\geq \lambda t + x \biggl(
\frac{t}{\alpha x} \biggr)^{\frac{\alpha}{\alpha-1}} - t \biggl(\frac{t}{\alpha x}
\biggr)^{\frac{1}{\alpha-1}} - \lambda ^{\alpha}x
\nonumber
\\
& = \lambda t - \lambda^{\alpha}x + (1-\alpha) \biggl(\frac{\alpha
}{t}
\biggr)^{\frac{\alpha}{1-\alpha}} x^{\frac{1} {1-\alpha}}
\\
& = \lambda t - \lambda^{\alpha}x + d(\alpha,t)
x^{\frac{1} {1-\alpha }}\quad  \mbox(\textrm{say}),
\end{align}
It follows that
%
%e8 #&#
\begin{equation}
\lim_{x\rightarrow\infty}\frac{-\log\mathbb{P}(E_{\alpha,\lambda
}(t)>x)}{x\log x} = \infty.
\end{equation}
Using the same argument as in Proposition \xch{\ref{non-id-iss}}{\eqref{non-id-iss}}, we
conclude that distributions of ITSS are not ID.
\end{proof}

Next we discuss the non-infinite divisibility\querymark{Q7} of the distribution of
inverse of an inverse Gaussian subordinator. It is worth to mention
that an inverse Gaussian subordinator is a particular case of TSS. Let
$G(t)$ be an inverse Gaussian subordinator with parameters $\delta$
and $\gamma$, then its density function is given by
%
%e9 #&#
\begin{equation}
f_{G(t)}(y) = \frac{\delta t}{2\pi}e^{\delta\gamma t} y^{-3/2}
e^{-\frac{1}{2}(\frac{\delta^2t^2}{y}+\gamma^2 y)}.
\end{equation}
Further, the Laplace exponent for $G(t)$ is given by $\psi(u) = \delta
(\sqrt{2u+\gamma^2} -\gamma)$ (see Applebaum, \cite{App2009}, p. 54). Let $H(t)
= \inf\{s\geq0: G(s)>t \}$ be the first-passage time process. Using
\eqref{tail-prob}, it follows
%
%e10 #&#
\begin{equation}
\mathbb{P}\bigl(H(t)>x\bigr) \leq e^{-\frac{\delta^2x^2}{t} + \delta\gamma
x-\frac{\gamma^2 t}{2}}.
\end{equation}
Using the similar argument as earlier, we have the following result.
%
%p3.3 #&#
\begin{proposition}
The distribution of the first-passage time process $H(t)$ is not
ID.
\end{proposition}
%

%
%r3.2 #&#
\begin{remark}
Note that when $\gamma=0$, the distribution of $H(t)$ is folded
Gaussian, which is not ID; the latter is a known result (see, e.g.,
Steutel and Van Harn, \cite{SteHar2004}, p.~126).
\end{remark}
%

%
%r3.3 #&#
\begin{remark}
A proof of non-infinite divisibility of distribution of $H(t)$ is
discussed in Vellaisamy and Kumar \cite{VelKum2018}, where the tail probabilities' bound is
obtained by using different techniques.
\end{remark}
 Next, we discuss the tail probabilities for gamma
subordinators. Let $U(t)$ be the gamma subordinator with parameters
$a,b>0$, having the density function
%
%e11 #&#
\begin{equation}
f_{U(t)}(x) = \frac{b^{at}}{\varGamma(at)} x^{at-1} e^{-bx},\quad
x>0.
\end{equation}
The Laplace exponent for the gamma subordinator is given by $\psi(u) = a \log
(1+\frac{u}{b})$ (see Applebaum, \cite{App2009}, p. 55), which implies $\psi
'^{-1}(u) = \frac{a-bu}{u}$. Let $V(t)$ be the first-passage time of
$U(t)$, then using \eqref{tail-prob}
%
%e12 #&#
\begin{equation}
\mathbb{P}\bigl(V(t) > x\bigr) \leq \biggl(\frac{bt}{ax}
\biggr)^{ax} e^{ax-bt},\quad  \mbox{for large $x$}.
\end{equation}
It follows that $\lim_{x\rightarrow\infty}\frac{-\log\mathbb{P}(V(t)
>x)}{x\log x} \geq a$ and hence unlike Proposition \ref{non-id-iss}
there is no obvious contradiction. So, we can't say anything about the
infinite divisibility of first-exit times of gamma subordinators. In
this article, we are not able to conclude whether inverse of a gamma
subordinator has ID marginals or not. It is worth to
mention that inverse Gaussian distributions or, more generally,
generalized inverse Gaussian (GIG) distributions are generalized gamma
convolutions (Halgreen, \cite{Hal1979}) and hence are ID. The inverse
gamma subordinator and the first-exit times of a gamma subordinator are
different processes. The density of an inverse gamma subordinator is a
particular case of GIG densities which are ID.

 Next we discuss some transformed processes of the inverse
subordinators. Consider the transformed ISS $E(t)^p,\; p>0$. We have
%
%e13 #&#
\begin{align}
\label{transformed-ISS} \mathbb{P}\bigl(E(t)^p >x\bigr) &= \mathbb{P}\bigl(E(t)
>x^{1/p}\bigr)
\nonumber
\\
& = \mathbb{P}\bigl(T\bigl(x^{1/p}\bigr) < t\bigr) = \mathbb{P}
\bigl(e^{-uT(x^{1/p})} \geq e^{-ut}\bigr)
\nonumber
\\
& \leq e^{ut-x^{1/p}\psi(u)}, \quad u>0
\nonumber
\\
& \leq e^{t \psi'^{-1}(t/x^{1/p}) - x^{1/p} \psi(\psi
'^{-1}(t/x^{1/p}))},\quad \mbox{for large $x$}.
\end{align}
Using \eqref{transformed-ISS} and the similar argument as in \xch{Propositions
\ref{non-id-iss} and \ref{non-id-itss}}{propositions
\eqref{non-id-iss} and \eqref{non-id-itss}}, we have the following result.
%
%p3.4 #&#
\begin{proposition}
The transformed ISS $E_{\alpha}(t)^p$ and transformed ITSS $E_{\alpha
, \lambda}(t)^p$ do not have ID distributions for $p
< 1/(1-\alpha)$. Transformed first-passage times of inverse Gaussian
subordinators defined by $H(t)^p$ do not have ID
distributions for $p<2$. Further, transformed first-passage times of
gamma subordinators defined by $V(t)^p$ do not have infinitely
divisible distributions for $p<1$.
\end{proposition}
\begin{proof}
We here provide the proof for an inverse gamma subordinator only. Proofs
for other subordinators follow similarly. Note that
\begin{align*}
\mathbb{P}\bigl(V(t)^p >x\bigr) \leq \biggl(\frac{bt}{ax^{1/p}}
\biggr)^{ax^{1/p}} e^{ax^{1/p}-bt},\quad  \mbox{for large $x$}.
\end{align*}
Thus $-\log\mathbb{P}(V(t)^p >x) \geq ax^{1/p} \log(a) + \frac
{a}{p}x^{1/p}\log x + bt - ax^{1/p} \log(bt)-ax^{1/p}$, which implies
%
%e14 #&#
\begin{equation}
\lim_{x\rightarrow\infty}\frac{-\log\mathbb{P}(V(t)>x)}{x\log x} = \infty,\quad \mbox{for}\; 0<p<1,
\end{equation}
and hence by the necessary condition \eqref{Steutel-VanHarn}, $V(t)^p$
does not have ID distribution for $0<p<1$.
\end{proof}

\setcounter{equation}{0}
%s4 #&#
\section{Compositions of ISS}
We can easily show that $E_{\alpha}(t)$ is self-similar with
self-similarity index $\alpha$. Note that
\begin{align*}
\mathbb{P}\bigl(E_{\alpha}(ct) \leq x\bigr) &= \mathbb{P}
\bigl(S_{\alpha}(x) \geq ct\bigr) = \mathbb{P} \biggl(\frac{1}{c}S_{\alpha}(x)
\geq t \biggr)
\\
& = \mathbb{P}\bigl(S_{\alpha}\bigl(x/c^{\alpha}\bigr) \geq t\bigr)
= \mathbb{P} \biggl(E_{\alpha}(t) \leq\frac{x}{c^{\alpha}} \biggr)
\\
& = \mathbb{P}\bigl(c^{\alpha}E_{\alpha}(t) \leq x\bigr),
\end{align*}
and hence $E_{\alpha}(ct) \stackrel{d}= c^{\alpha}E(t)$.

For a strictly increasing subordinator $T(t)$ with the Laplace exponent
$\psi(u)$, the density function $q(x,t)$ of the inverse subordinator
has the LT with respect to time variable (see Meerschaert and
Scheffler, \cite{MeeSch2008})
\begin{equation*}
\mathcal{L}_t\bigl(q(x,t)\bigr) =\int_0^\infty
e^{-st}q(x,t)dt= \frac{1}{s} \psi(s) e^{-x \psi(s)}.
\end{equation*}
Let $g(x,t)$ be the density function of the ISS $E_{\alpha}(t)$. Then
$\mathcal{L}_t(g(x,t)) = s^{\alpha-1}e^{-xs^{\alpha}}$. Let $E^*(t)
= E_{\alpha_1}(E_{\alpha_2}(t))$ represent the composition of two
independent inverse stable subordinators. Further, let $h(x,t)$ be the
density function of $E^*(t)$ and let $h_1(x,t)$ and $h_2(x,t)$ be the
density functions of $E_{\alpha_1}(t)$ and $E_{\alpha_2}(t)$
respectively. Then
\begin{align*}
h(x,t) = \int_{0}^{\infty}h_1(x,r)h_2(r,t)dr.
\end{align*}
Thus
\begin{align*}
\mathcal{L}_t\bigl(h(x,t)\bigr) &= \int_{0}^{\infty}h_1(x,r)
\mathcal {L}_t\bigl(h_2(r,t)\bigr)dr
\\
& = \int_{0}^{\infty}h_1(x,r)
s^{\alpha_2-1}e^{-rs^{\alpha_2}} dr
\\
& = s^{\alpha_2-1} \int_{0}^{\infty}h_1(x,r)e^{-rs^{\alpha_2}}
dr
\\
& = s^{\alpha_2-1} \bigl(s^{\alpha_2}\bigr)^{\alpha_1-1}
e^{-xs^{\alpha
_1\alpha_2}}
\\
& = s^{\alpha_1\alpha_2 -1} e^{-xs^{\alpha_1\alpha_2}}.
\end{align*}
Hence $E^*(t) = E_{\alpha_1}\circ E_{\alpha_2}(t) = E_{\alpha
_1}(E_{\alpha_2}(t))$ is the same in distribution as an ISS of index
$\alpha_1\alpha_2$. In general, let $E_{\alpha_1}(t), E_{\alpha
_2}(t),\ldots, E_{\alpha_n}(t)$ be independent ISS with indices
$\alpha_1, \alpha_2,\ldots, \alpha_n$ respectively. Then the
process defined by the composition $E^*(t) = E_{\alpha_1}\circ
E_{\alpha_2}\circ\cdots\circ E_{\alpha_n}(t)$ is the same in
distribution as an ISS with index $\alpha_1\alpha_2\cdots\alpha_n$.
Further, the distribution of the process $E^*(t)$ is not infinitely
divisible. Next, we prove the non-infinite divisibility of the
distribution of a time-changed ISS where the time-change is a general subordinator.

%r4.1 #&#
\begin{remark}
Nane \cite{Nan2010} has considered the composition of independent inverse stable
subordinators of index $\alpha=1/2$. He observed that for fixed $t\geq
0$, the $k$-iterated Brownian motion
\[
|I_{k}(t)|=\big|B_1\bigl(\big|B_2\bigl(\big|\cdots
\bigl(|B_{k}(t)|\bigr)\cdots\big|\bigr)\big|\bigr)\big|
\]
and $E^{1/2^k}(t)=E_{1/2}\circ E_{1/2}(t)\cdots\circ E_{1/2}(t)$ have
the same one-dimensional distributions.
\end{remark}

%p4.1 #&#
\begin{proposition}
Let $T(t)$ be a general subordinator with finite mean i.e. a positive
L\'evy process with non-decreasing sample paths having $\mathbb
{E}(T(1)) < \infty$. Then the time-changed process $E_{\alpha}(T(t))$
does not have ID distributions.
\end{proposition}
\begin{proof}
By self-similarity of $E_\alpha(t)$ we have
\begin{align*}
\frac{E_{\alpha}(T(t))}{t^{\alpha}} &\stackrel{d} = \biggl(\frac
{T(t)}{t}
\biggr)^{\alpha} E_{\alpha}(1)
\\
& \stackrel{a.s.}\longrightarrow\bigl(\mathbb{E}\bigl(T(1)\bigr)
\bigr)^{\alpha
}E_{\alpha}(1),
\end{align*}
as $t\to\infty$. Here we have used the fact that for a subordinator
$T(t)/t \rightarrow\mathbb{E}(T(1))$ a.s. as $t\rightarrow\infty$
(see, e.g., Bertoin, \cite{Ber1996}, p. 92). Since the a.s. convergence
implies the convergence in distribution (see, e.g., Chung, \cite{Chu2001}), it follows
$\frac{E_{\alpha}(T(t))}{t^{\alpha}} \,{\stackrel{d} \rightarrow}\,
(\mathbb{E}(T(1)))^{\alpha}E_{\alpha}(1)$. Assume that $E_{\alpha
}(T(t))$ has an ID distribution, then $E_{\alpha
}(T(t))/t^{\alpha}$ will also have an ID distribution
for each $t$ (see, e.g., Steutel and Van Harn, \cite{SteHar2004}, Prop. 2.1, p. 94).
Next we recall Prop. 2.2 from Steutel and Van Harn (\cite{SteHar2004}, p.~94): If
a sequence of $\mathbb{R}_+$-valued
random variables $X_n, n\geq0$ with ID distributions
converges in distribution to $X$, then $X$ has ID distribution.
Hence, the limit in distribution i.e. $(\mathbb{E}(T(1)))^{\alpha
}E_{\alpha}(1)$ will also have an ID distribution or,
equivalently, $E_{\alpha}(1)$ will also have an ID
distribution. This is a contradiction.
\end{proof}

%c4.1 #&#
\begin{corollary}
Let $S_{\alpha, \lambda}(t), G(t)$ and $U(t)$ be tempered stable,
inverse Gaussian and gamma subordinators. Then the distributions of
time-changed processes $E_{\alpha}(S_{\alpha, \lambda}(t))$,
$E_{\alpha}(G(t)) $ and $E_{\alpha}(U(t))$ are not infinitely
divisible. Further, $E^*(S_{\alpha, \lambda}(t))$, $E^*(G(t)) $ and
$E^*(U(t))$ also do not have ID distributions, where
$E^*$ is the composition of $n$ independent inverse stable subordinators.
\end{corollary}

%{\color{red} By the result in Bertoin (\cite{Ber1996}) we cannot conclude the
%result in corollary for $E_{\alpha}(S_{\alpha}(t))$ and for$E^*(S_{
%\alpha}(t))$ as the $\mathbb{E}(S_\alpha(1))=\infty$}

\setcounter{equation}{0}
%s5 #&#
\section{Time-changed renewal processes}

Let $W_i,\; i=1,2,\ldots $ be a sequence of i.i.d. a.s. positive random
variables. Then the random
walk
$T_0 =0, T_n = W_1+ \cdots+ W_n$, $n \geq1$, is said to be a renewal
sequence and the counting process
$N(t) = \max\{i: T_i\leq t\}$ is called the corresponding renewal
process (see, e.g., Mikosch, \cite{Mik2009}, p. 59). We have the following result for
the time-changed renewal process.

%p5.1 #&#
\begin{proposition}\label{non-id-renewal}
Let $N(t)$ be a renewal process with finite expectations of the inter-arrival times $W_i$, $\mathbb{E}W_1 =
\lambda^{-1}$. The
renewal process time-changed by an ISS defined by $N(E_{\alpha}(t))$
does not have ID distribution.
\end{proposition}
\begin{proof}
Note that $E_{\alpha}(1)> 0$ a.s. Hence, as $t\rightarrow
\infty$, it follows $t^{\alpha}E_{\alpha}(1) \rightarrow\infty$
a.s. By an application of the renewal theorem, we have
\begin{align*}
\lim_{t\rightarrow\infty} \frac{N(t^{\alpha}E_{\alpha
}(1))}{t^{\alpha}} &= \lim_{t\rightarrow\infty}
\frac{N(t^{\alpha
}E_{\alpha}(1))}{t^{\alpha}E_{\alpha}(1)} E_{\alpha}(1)
\\
&\stackrel{a.s.}\longrightarrow\lambda E_{\alpha}(1).
\end{align*}
Further due to self-similarity of $E_{\alpha}(t)$, it follows that
$N(E_{\alpha}(t))\stackrel{d} = N(t^{\alpha}E_{\alpha}(1))$ and
hence $N(E_{\alpha}(t))/t^{\alpha}\stackrel{d}\rightarrow\lambda
E_{\alpha}(1)$. Suppose that $N(E_{\alpha}(t))$ has infinitely
divisible distribution then $N(E_{\alpha}(t))/t^{\alpha}$ will also
have an ID distribution (see, e.g., Steutel and Van
Harn, \cite{SteHar2004}, Prop. 2.1, p. 94). Since the limit of a sequence of random
variables with ID distributions has an infinitely
divisible distribution (see, e.g., Sato \cite{Sat1999}; Steutel and Van Harn,
\cite{SteHar2004}), we have that $E_{\alpha}(1)$ has an ID
distribution, and hence contradiction.
\end{proof}

Meerschaert et al. \cite{Meeetal2009} establish that the fractional Poisson process
introduced by Laskin \cite{Las2003} can be obtained from the standard
Poisson process by the time-change with an ISS. Since the Poisson process is a renewal process, we
have the following result.
%
%c5.1 #&#
\begin{corollary}
Let $M(t)$ be the standard Poisson process. Then the fractional Poisson
process which is defined by $M^*(t) = M(E_{\alpha}(t))$ does not have
an ID distribution. Further, since $M^*(t)$ is also a
renewal process with Mittag-Leffler waiting times, the time-changed
process defined by $M^{**}(t) = M^*(E_{\beta}(t))$ does not have an
ID distribution.
\end{corollary}

%\begin{appendix}
%\end{appendix}

%\begin{acknowledgement}%[title={Acknowledgments}]
%\end{acknowledgement}

%\begin{funding}
%\gsponsor[id=,sponsor-id=]{}
%\gnumber[refid=]{}
%\end{funding}

% structpyb loaded by romualda, 2018-07-16 13:44:50
% structpyb loaded by romualda, 2018-07-16 14:19:48

\end{document}